\documentclass[12pt,a4paper]{article}
\usepackage[margin=25mm,headheight=15pt]{geometry}
\usepackage[T1]{fontenc}
\usepackage{lmodern,amsmath,amssymb,amsthm,mathtools,microtype}
\usepackage[hidelinks]{hyperref}
\usepackage{fancyhdr,needspace}
\theoremstyle{plain}
\newtheorem{theorem}{Theorem}[section]
\newtheorem{lemma}[theorem]{Lemma}
\newtheorem{proposition}[theorem]{Proposition}
\newtheorem{remark}[theorem]{Remark}
\newtheorem{example}[theorem]{Example}
\newtheorem*{remark*}{Remark}
\newtheorem*{lemma*}{Lemma}
\numberwithin{equation}{section}
\DeclareMathOperator{\sgn}{sgn}

\newcommand{\proofstep}[1]{\leavevmode\par\medskip\noindent\textbf{#1}\enspace}
\allowdisplaybreaks[2]
\title{Blow-up, critical norm inflation, and global existence\\for the Geng--Xue system}
\author{Song Liu\thanks{Email: lius37@mail2.sysu.edu.cn}\qquad Zhaoyang Yin\thanks{Email: mcsyzy@mail.sysu.edu.cn}\\[5pt]
\small School of Science, Shenzhen Campus of Sun Yat-sen University,\\
\small Shenzhen 518107, China}
\date{}
\begin{document}
\maketitle
\begin{abstract}
We study finite-time blow-up, critical norm inflation, and global
well-posedness for the Geng--Xue system, an integrable two-component
extension of the Novikov equation.
For initial velocities $(u_0,v_0)\in(H^s)^2$, $s>3/2$,
with second momentum $n_0=v_0-v_{0,xx}\in L^1_{\rm uloc}$,
we derive explicit blow-up conditions and upper bounds on the lifespan.
The blow-up proof uses a cancellation in the equation for $nu_x$
to close the convolution estimates without an a priori bound
on $\|u_x\|_{L^\infty}$.
Combined with a continuation criterion involving only $u$,
these estimates yield critical Besov norm inflation
around the initial data in $(H^{3/2})^2$.
For initial momenta in $(H^\sigma)^2$, $\sigma>-1/2$,
each with a fixed sign, we prove global well-posedness.
The blow-up and norm-inflation results also extend to the
three-component Degasperis--Procesi system.
\end{abstract}

\clearpage
\section{Introduction}

We consider the Cauchy problem for the Geng--Xue (GX) system
\cite{r7}:
\begin{equation}\label{gx}
\begin{cases}
m_t+uvm_x+3vu_xm=0,\\[2pt]
n_t+uvn_x+3uv_xn=0,\\[2pt]
m=u-u_{xx},\qquad n=v-v_{xx}.
\end{cases}
\end{equation}
This system was introduced by Geng and Xue \cite{r7},
who established its Hamiltonian structure and showed that
it admits peakon solutions.
It is an integrable two-component extension of the Novikov
equation and is closely related to the Camassa--Holm and
Degasperis--Procesi shallow-water models
\cite{GXPeakonDynamics}.
Lundmark and Szmigielski \cite{r16} solved the associated
spectral and inverse spectral problems.

For the Cauchy problem, Himonas and Mantzavinos \cite{r10}
proved local well-posedness in $(H^s)^2$, $s>3/2$,
and failure of uniform continuity of the data-to-solution map.
Chen, Qiao, and Zhou \cite{r2} studied persistence properties
and obtained a blow-up criterion.
Liu and Yin \cite{r15} derived a blow-up criterion in Besov
spaces of low regularity and proved finite-time blow-up
under pointwise assumptions on the initial data,
without using conservation laws.

When $v\equiv1$, \eqref{gx} reduces to the
Degasperis--Procesi (DP) equation \cite{r5}:
\[
m_t+um_x+3u_xm=0,\qquad m=u-u_{xx}.
\]
The DP equation is related to asymptotic models for
unidirectional shallow-water waves \cite{r3};
the equation above is its dispersionless form.
It has a bi-Hamiltonian structure, infinitely many
conservation laws, and peakon solutions \cite{r4}.
Parametric solutions of the DP equation and
algebro-geometric solutions of its hierarchy were
constructed in \cite{r40,r38}, respectively.
For results on the Cauchy and initial-boundary value
problems, see \cite{r32,r34,r41,r42}.

A related shallow-water model is the Camassa--Holm (CH)
equation \cite{r27,r31}, whose dispersionless form is
\[
m_t+um_x+2u_xm=0,\qquad m=u-u_{xx}.
\]
Camassa and Holm \cite{r27} derived the corresponding
model by an asymptotic expansion of the Euler Hamiltonian.
Like the DP equation, it has a bi-Hamiltonian structure
and admits peakon solutions.
Both equations exhibit wave breaking, in which the velocity
remains bounded while its spatial derivative becomes
unbounded \cite{r3}.
For results on the Cauchy problem for the CH equation,
see \cite{r29,r28}.

When $v\equiv u$, \eqref{gx} reduces to the Novikov equation
\cite{r18}:
\[
m_t+u^2m_x+3uu_xm=0,\qquad m=u-u_{xx}.
\]
This integrable equation has a cubic nonlinearity and admits
peakon solutions \cite{r37}.
Unlike the CH and DP equations, its transport velocity is
$u^2$, so the transport speed depends quadratically on
the wave amplitude.
For results on well-posedness, blow-up, and ill-posedness,
see \cite{r35,r36,r39}.

We also extend the blow-up and norm-inflation results to the
three-component Degasperis--Procesi (3DP) system:
\begin{equation}\label{3dp}
\begin{cases}
\rho_t+(\rho uv)_x=0,\\[2pt]
m_t+uvm_x+3vu_xm+\rho^2u=0,\\[2pt]
n_t+uvn_x+3uv_xn-\rho^2v=0,\\[2pt]
m=u-u_{xx},\qquad n=v-v_{xx}.
\end{cases}
\end{equation}
When $\rho\equiv0$, this system reduces to the GX system
\eqref{gx}.
Here $\rho$ satisfies a continuity equation with transport
velocity $uv$ and introduces source terms into the momentum
equations.
For related three-component integrable systems, their local
well-posedness, and their reductions, see \cite{r12,r13,r14}.
Our argument incorporates the density terms without assuming
$\rho_0^2-\bar\rho^2\in L^1(\mathbb R)$,
where $\bar\rho$ is the constant density background.

\subsection[Notation]{Notation}\label{sec:notation}
\begingroup
Throughout this paper, all functions are real-valued.
Function spaces are nonhomogeneous and, unless stated otherwise,
are defined on $\mathbb R$. Product norms are sums of the component
norms. We denote the Schwartz class by $\mathcal S(\mathbb R)$
and the circle by $\mathbb T$.
Let $D^a=(1-\partial_x^2)^{a/2}$ and
$P(x)=\tfrac12e^{-|x|}$, the Green function of $1-\partial_x^2$.
We set $\eta=\rho-\bar\rho$, where $\bar\rho$ is the constant
density background.

We use the uniformly local Lebesgue space
$L^1_{\rm uloc}(\mathbb R)$; see, for example,
\cite[p.~270]{BerselliBisconti2015}.
Specifically,
\[
L^1_{\rm uloc}
=\left\{f\in L^1_{\rm loc}:\|f\|_{L^1_{\rm uloc}}
:=\sup_{a\in\mathbb R}\int_a^{a+1}|f(x)|\,dx<\infty\right\}.
\]
Constants may change from line to line, and their dependence
is indicated where needed. The notation $A\lesssim B$ means
$A\le CB$ for such a constant $C>0$.

For the blow-up criterion, we use the following quantities
associated with a solution of \eqref{3dp}.
Set $g=nu_x+\rho^2-\bar\rho^2$ and
\begin{equation}\label{blowup-quantities}
A(t)=\|P*(|u(t)|+|u_x(t)|)\|_{L^\infty},\quad
B(t)=\|P*|n(t)|\|_{L^\infty},\quad
N(t)=\|P*|g(t)|\|_{L^\infty}.
\end{equation}
Their initial values are denoted by $A_0$, $B_0$, and $N_0$,
{respectively. Let} $K=(\bar\rho^2+A_0^2+B_0^2+N_0)^{1/2}$.
{For $K>0$, $\nu>0$, and $\theta\in(0,1)$, set}
\begin{equation}\label{threshold-coeff}
S_\theta=\left(K+\tfrac53(1-\theta)\nu\right)^2,\quad
T_\theta=\frac{S_\theta-K^2}{10S_\theta K^2},\quad
d_\theta=\sqrt{S_\theta^2/4+8\theta\nu S_\theta^{3/2}}.
\end{equation}
\endgroup

\subsection{Main results}
\begin{theorem}[Finite-time blow-up]\label{thm:blowup}
Let $s>3/2$ and $\bar\rho\in[0,1]$. Assume that
$(\rho_0-\bar\rho,u_0,v_0)\in H^{s-1}\times(H^s)^2$
and $n_0\in L^1_{\rm uloc}$, and let $(\rho,u,v)$ be the
corresponding maximal strong solution of \eqref{3dp},
with lifespan $T^*$.
Fix $x_0\in\mathbb R$ such that $v_0(x_0)\ne0$, and set
$\nu=|v_0(x_0)|$ and
$y_0=-\sgn(v_0(x_0))u_{0,x}(x_0)$.
With the notation introduced above, if
\[
y_0\ge\Phi_\theta(K,\nu)
=\frac{S_\theta+2d_\theta\coth(d_\theta T_\theta/2)}{4\theta\nu}
\]
for some $\theta\in(0,1)$, then the solution blows up
in finite time, with
\[
T^*\le\tau_\theta(K,\nu,y_0)
=\frac1{d_\theta}\log\frac{4\theta\nu y_0-S_\theta+2d_\theta}
{4\theta\nu y_0-S_\theta-2d_\theta}
\le T_\theta<\infty.
\]
\end{theorem}
\begin{remark}
For $s\ge2$, the momentum assumption follows from $n_0\in H^{s-2}\subset L^2\subset L^1_{\rm uloc}$, with $\|n_0\|_{L^1_{\rm uloc}}\le\|n_0\|_{L^2}$. The initial convolution quantities are finite: $A_0\le\|u_0\|_{W^{1,\infty}}$, $B_0\le C\|n_0\|_{L^1_{\rm uloc}}$ and $N_0\le\|u_{0,x}\|_{L^\infty}B_0+\|\rho_0^2-\bar\rho^2\|_{L^\infty}$. Setting $\rho_0=\bar\rho=0$ gives the GX case.

For $k,\nu>0$, define $\Phi_\theta(k,\nu)$ and ${\tau_\theta}(k,\nu,y)$ {by replacing $K$ with $k$ in \eqref{threshold-coeff} and the formulas in Theorem~\ref{thm:blowup}, and $y_0$ with $y$ in the lifespan formula.} {The threshold and lifespan bound remain valid for any $k\ge K$}, and $\min_{0<\theta<1}\Phi_\theta(K,\nu)$ is attained in $(0,1)$. The reciprocal scaling $(u,v,\rho)\mapsto(\lambda u,\lambda^{-1}v,\rho)$ yields the sufficient condition for finite-time blow-up
\begin{equation}\label{scaling-criterion}
y_0>\inf_{\lambda>0,\ 0<\theta<1}\frac1\lambda
\Phi_\theta\left(\sqrt{\bar\rho^2+\lambda^2A_0^2+\lambda^{-2}B_0^2+N_0},\frac\nu\lambda\right).
\end{equation}
For $(\lambda,\theta)$ satisfying the strict threshold, the lifespan $T^*$ satisfies
\[
T^*\le{\tau_\theta}\left(\sqrt{\bar\rho^2+\lambda^2A_0^2+\lambda^{-2}B_0^2+N_0},\frac\nu\lambda,\lambda y_0\right).
\]
The coupled comparison system in Section~\ref{sec:blowup} gives a strictly smaller threshold and preserves this scaling invariance.
\end{remark}

\begin{theorem}[Norm inflation]\label{thm:inflation}
Let $\bar\rho\in[0,1]$, $1\le p\le\infty$, $1<r\le\infty$, and
$(\eta_*,u_*,v_*)\in
B^{1/p}_{p,r}\times(B^{1+1/p}_{p,r})^2$.
If $p=\infty$ or $r=\infty$, assume that
$(\eta_*,u_*,v_*)$ is the limit of a sequence in
$\mathcal S(\mathbb R)^3$.
For every $0<\varepsilon<1$, there are initial data
$(\rho_{0,\varepsilon}-\bar\rho,u_{0,\varepsilon},
v_{0,\varepsilon})\in\mathcal S(\mathbb R)^3$ satisfying
\[
\|\rho_{0,\varepsilon}-\bar\rho-\eta_*\|_{B^{1/p}_{p,r}}
+\|u_{0,\varepsilon}-u_*\|_{B^{1+1/p}_{p,r}}
+\|v_{0,\varepsilon}-v_*\|_{B^{1+1/p}_{p,r}}
<\varepsilon
\]
such that the corresponding smooth solution of \eqref{3dp}
has maximal lifespan $T_\varepsilon^*<\varepsilon$ and,
for some $0<t_\varepsilon<T_\varepsilon^*$, satisfies
\[
\|u(t_\varepsilon)-u_*\|_{B^{1+1/p}_{p,r}}
>\varepsilon^{-1},
\qquad
\int_0^{T_\varepsilon^*}
\|u(t)\|_{B^{1+1/p}_{p,r}}\,dt=\infty.
\]
For GX, the same conclusions hold near any
$(u_*,v_*)\in(B^{1+1/p}_{p,r})^2$,
with the same Schwartz approximation assumption
when $p=\infty$ or $r=\infty$.
In this case, $\rho\equiv0$ and the density term is omitted
from the initial distance.
\end{theorem}
\begin{remark}
For $p,r<\infty$, Schwartz functions are dense in the
stated Besov spaces \cite{r1}, so the approximation
assumption is automatic.
In particular, the theorem applies to all initial data in
$H^{1/2}\times(H^{3/2})^2$ and, for GX system, $(H^{3/2})^2$.
The integral divergence implies
\[
\limsup_{t\uparrow T_\varepsilon^*}
\|u(t)-u_*\|_{B^{1+1/p}_{p,r}}=\infty.
\]
No sign condition is imposed on the perturbed momenta.
The isomorphism $1-\partial_x^2$ gives the equivalent
momentum formulation in
$B^{1/p}_{p,r}\times(B^{-1+1/p}_{p,r})^2$,
with the same Schwartz approximation requirement.
\end{remark}
\begin{theorem}[Global well-posedness]\label{thm:global}
Let $\sigma>-1/2$ and $(m_0,n_0)\in(H^\sigma(\mathbb R))^2$.
Assume that $\varepsilon_1m_0\ge0$ and $\varepsilon_2n_0\ge0$
in the sense of distributions, for some
$\varepsilon_1,\varepsilon_2\in\{-1,1\}$.
Then \eqref{gx} has a unique global strong solution
\[
(m,n)\in C(\mathbb R;(H^\sigma)^2)
\cap C^1(\mathbb R;(H^{\sigma-1})^2),
\]
with $\varepsilon_1m(t)\ge0$ and $\varepsilon_2n(t)\ge0$
for all $t\in\mathbb R$.
For every $T>0$, the solution map is continuous
from the initial data with these prescribed signs,
equipped with the $(H^\sigma)^2$ norm,
into $C([-T,T];(H^\sigma)^2)$.

For $-1/2<\sigma\le0$, the solution satisfies
\begin{equation}\label{negative-bound}
\begin{aligned}
\|m(t)\|_{H^\sigma}+\|n(t)\|_{H^\sigma}
&\le
\exp\left(
2|t|\left|\int_{\mathbb R}
(u_0v_0+u_{0,x}v_{0,x})\,dx\right|
\right)\\
&\quad\times
\bigl(\|m_0\|_{H^\sigma}+\|n_0\|_{H^\sigma}\bigr),
\qquad t\in\mathbb R.
\end{aligned}
\end{equation}
\end{theorem}
\begin{remark}
For GX data satisfying the assumptions of
Theorem~\ref{thm:blowup}, the sign condition rules out
the stated blow-up condition.
Indeed, if $m_0$ has a fixed sign, then $|P_x|\le P$ gives
\[
y_0\le\|u_{0,x}\|_{L^\infty}
\le\|u_0\|_{L^\infty}\le A_0\le K.
\]
Since $0<\nu\le B_0\le K$ and $S_\theta\ge K^2$,
\[
\Phi_\theta(K,\nu)
>\sqrt{\frac{2S_\theta^{3/2}}{\theta\nu}}
\ge\sqrt{2}\,K>K,
\qquad 0<\theta<1.
\]
Hence $y_0\ge\Phi_\theta(K,\nu)$ cannot hold.
\end{remark}

\subsection{Main ideas}\label{sec:contributions}
In our earlier work \cite{r15}, we estimated $\|u_x\|_{L^1}$
by testing the slope equation against a regularized sign
function.
Integration by parts in the transport term $uvu_{xx}$
cancels the contribution of $vu_x^2$.
The nonlocal terms were then controlled using
$\|u\|_{W^{1,1}}$ and $\|n\|_{L^\infty}$,
with the latter bound obtained by Sobolev embedding
for $s>5/2$.

We estimate the product $nu_x$ in this paper.
This choice follows from rewriting the nonlocal term
$3uvu_x-uu_xv_{xx}$ as $2uvu_x+u(nu_x)$ using $v_{xx}=v-n$.
In the conservative equation for $nu_x$, the term
$-vnu_x^2$ cancels with its counterpart in $(uv)_xnu_x$,
leaving the local term $-uv_xnu_x$.
This cancellation allows us to close the integral estimates
for $A$, $B$, and $N$ defined in \eqref{blowup-quantities}.
Here $A$ bounds the convolutions involving $u$ and $u_x$,
and $B$ bounds both $v$ and $v_x$ because $v=P*n$ and
$|P_x|\le P$.
For GX, $N$ measures the convolution of $|nu_x|$.

Thus, these estimates bound the nonlocal terms in
the slope equation without an a priori bound on
$\|u_x\|_{L^\infty}$ or $\|n\|_{L^\infty}$.
The quadratic term $-vu_x^2$ remains in that equation
and yields a Riccati inequality along characteristics.
An approximation argument extends the result to $s>3/2$
with $n_0\in L^1_{\rm uloc}$, removing the global
$W^{1,1}$ assumption on $u_0$ used in \cite{r15}.

For 3DP, adding $\rho^2$ to $nu_x$ also cancels the
local term involving $\rho^2vu_x$ in the conservative
equation.
Subtracting the constant background gives
$g=nu_x+\rho^2-\bar\rho^2$, which replaces $nu_x$
in the definition of $N$.
The resulting estimates require no global $L^1$
assumption on $\rho_0^2-\bar\rho^2$.

To prove norm inflation, we perturb background data
satisfying the assumptions of Theorem~\ref{thm:inflation}
by functions with small critical norms and large slopes,
while keeping the initial convolution quantities
uniformly bounded.
The blow-up result gives an explicit upper bound on
the lifespan that tends to zero along this sequence.
A continuation criterion involving only $u$ then
forces the critical Besov norm of $u$ to become
arbitrarily large.

For momenta with prescribed signs, we extend the global
well-posedness result of \cite{r22} from $\sigma>1/2$
to $\sigma>-1/2$.
The conserved mixed energy bounds the transport coefficients,
and positivity of the Bessel kernel gives the $H^\sigma$
estimates for $-1/2<\sigma<0$.
The two momentum signs may differ, and no smallness or
global $L^1$ assumption is required.

The paper is organized as follows.
Section~\ref{sec:prelim} collects the lemma and estimates
used in the proofs.
Section~\ref{sec:local} establishes local well-posedness
and continuation criteria.
The proofs of the blow-up, norm-inflation, and global
well-posedness results are presented in
Sections~\ref{sec:blowup}--\ref{sec:global}, respectively.
Section~\ref{sec:examples} gives examples illustrating
the blow-up and norm-inflation results.

\section{Preliminaries}\label{sec:prelim}
{We collect the Littlewood--Paley and transport estimates used below. All are stated in nonhomogeneous spaces and have periodic versions on $\mathbb T$.}

{Fix a standard nonhomogeneous dyadic decomposition $f=\sum_{j\ge-1}\Delta_jf$ and set $S_jf=\sum_{-1\le k<j}\Delta_kf$. The associated Besov norm is}
\[\|f\|_{B^a_{p,r}}=\big\|(2^{ja}\|\Delta_jf\|_{{L^p}})_{j\ge-1}\big\|_{\ell^r}.\]
We use Bernstein's inequalities and Bony's decomposition $fg=T_fg+T_gf+R(f,g)$, where $T_fg=\sum_jS_{j-1}f\Delta_jg$ and $R(f,g)=\sum_{|j-k|\le1}\Delta_jf\Delta_kg$. {For the paraproduct, remainder, and embedding estimates, we refer to \cite{r1}.}
\begin{lemma}[\cite{r11}]\label{lem:commutator}
{For $s>0$, there is a constant $c_s>0$ such that}
\[\|[D^s,f]g\|_{{L^2}}\le c_s\bigl(\|D^sf\|_{{L^2}}\|g\|_{{L^{\infty}}}+\|f_x\|_{{L^{\infty}}}\|D^{s-1}g\|_{{L^2}}\bigr).\]
\end{lemma}
\begin{lemma}[\cite{r10,r9}]\label{lem:product}
{For $\sigma>1/2$, there is a constant $c_\sigma>0$ such that}
\[\|fg\|_{H^{\sigma-1}}\le c_\sigma\|f\|_{H^\sigma}\|g\|_{H^{\sigma-1}}.\]
\end{lemma}
\begin{lemma}[\cite{r8}]\label{lem:friedrichs}
{Let $g_x\in L^\infty$ and let $J_\epsilon$ be a Friedrichs mollifier. There is a constant $C>0$, independent of $\epsilon$, such that}
{\[\|[J_\epsilon,g]\partial_xf\|_{L^2}\le C\|\partial_xg\|_{L^\infty}\|f\|_{L^2},\qquad f\in L^2.\]}
\end{lemma}
\begin{lemma}[\cite{r19}]
{Let $a>3/2$ and $0\le\sigma+1\le a$. For $f\in H^a$ and $g\in H^\sigma$,}
{\[\|[D^\sigma\partial_x,f]g\|_{L^2}\lesssim\|f\|_{H^a}\|g\|_{H^\sigma}.\]}
\end{lemma}
\begingroup
For $f_t+wf_x=S$, the transport estimates of \cite{r1} give, when $s>3/2$,
\[
\|f(t)\|_{H^{s-1}}\le\|f(0)\|_{H^{s-1}}
+\int_0^t\bigl(C_s\|w\|_{H^s}\|f\|_{H^{s-1}}+\|S\|_{H^{s-1}}\bigr)(\tau)\,d\tau.
\]
For $\sigma=0,1$ and $w\in H^2$,
\[
\|f(t)\|_{B^\sigma_{2,1}}\le\|f(0)\|_{B^\sigma_{2,1}}
+\int_0^t\bigl(C\|w\|_{H^2}\|f\|_{B^\sigma_{2,1}}+\|S\|_{B^\sigma_{2,1}}\bigr)(\tau)\,d\tau.
\]
{The second estimate applies because $w_x\in H^1\hookrightarrow B^{1/2}_{2,1}\cap L^\infty$.}
\endgroup
 Bony's decomposition also gives
\[
\|af\|_{B^0_{2,1}}\le C\|a\|_{H^1}\|f\|_{B^0_{2,1}},\qquad
\|fg\|_{B^1_{2,1}}\le C\bigl(\|f\|_{{L^{\infty}}}\|g\|_{B^1_{2,1}}+\|g\|_{{L^{\infty}}}\|f\|_{B^1_{2,1}}\bigr).
\]
For the first estimate one may use $H^1\hookrightarrow B^{1/2}_{\infty,2}$ in the paraproduct and remainder terms. {{Terms with constant coefficients are estimated separately.}}

\section{Local well-posedness and continuation criteria}\label{sec:local}

In this section, we first establish local well-posedness for $s>3/2$.
We then derive continuation criteria by showing that
certain integrals along characteristics must diverge
if the maximal existence time is finite.
\subsection{Local well-posedness}
\begin{theorem}[Local well-posedness]\label{thm:local}
Let $s>3/2$, $\bar\rho\in[0,1]$, and
$(\rho_0-\bar\rho,u_0,v_0)\in H^{s-1}\times(H^s)^2$,
where the spaces are defined on $\mathbb R$ or $\mathbb T$.
Set
\[
Y_0=1+\|\rho_0-\bar\rho\|_{H^{s-1}}
+\|u_0\|_{H^s}+\|v_0\|_{H^s}.
\]
There exists a constant $C_s>0$ such that, for
$T=(4C_sY_0^2)^{-1}$, system \eqref{3dp} has a unique
strong solution 
\[
(\rho-\bar\rho,u,v)
\in C([0,T];H^{s-1}\times(H^s)^2)
\cap C^1([0,T];H^{s-2}\times(H^{s-1})^2),
\]
satisfying
\[
1+\|\rho(t)-\bar\rho\|_{H^{s-1}}
+\|u(t)\|_{H^s}+\|v(t)\|_{H^s}
\le \sqrt{2}\,Y_0,\qquad 0\le t\le T.
\]
Moreover, the solution depends continuously on the initial
data in the above spaces.
\end{theorem}

\begin{remark}
Since $1-\partial_x^2$ is an isomorphism from
$H^{\sigma+2}$ onto $H^\sigma$,
Theorem~\ref{thm:local} gives local well-posedness
in the momentum variables for $\sigma>-1/2$, with
\[
(\rho-\bar\rho,m,n)
\in C([0,T];H^{\sigma+1}\times(H^\sigma)^2)
\cap C^1([0,T];H^\sigma\times(H^{\sigma-1})^2).
\]
Taking $\rho_0=\bar\rho=0$ gives the corresponding
GX result for $(m_0,n_0)\in(H^\sigma)^2$.

For $s=2$, the lower bound on the local existence time
implies that, if the maximal existence time $T^*$ is
finite, then
\begin{equation}\label{continuation-alt}
\limsup_{t\uparrow T^*}
\bigl(
\|\rho(t)-\bar\rho\|_{H^1}
+\|u(t)\|_{H^2}
+\|v(t)\|_{H^2}
\bigr)=\infty.
\end{equation}
\end{remark}
\begin{proof}[Proof of Theorem~\ref{thm:local}]
Following the Lagrangian approach used in the uniqueness
proof of \cite{r43}, we reduce the system to an ODE
with a locally Lipschitz vector field.
An energy estimate then gives the lifespan bound
and the continuation alternative.

We first write \eqref{3dp} as
\[
\eta_t+uv\eta_x=-(\bar\rho+\eta)(uv)_x,\qquad
u_t+uvu_x=-F,\qquad
v_t+uvv_x=-G,
\]
where
\begin{align*}
F(\eta,u,v)
&=D^{-2}\bigl\{
3uvu_x+(uu_xv_x)_x+u_x^2v_x
+uv_xu_{xx}+(\bar\rho+\eta)^2u
\bigr\},\\
G(\eta,u,v)
&=D^{-2}\bigl\{
3uvv_x+(vv_xu_x)_x+v_x^2u_x
+vu_xv_{xx}-(\bar\rho+\eta)^2v
\bigr\}.
\end{align*}
For $s>3/2$, the algebra property of $H^{s-1}$
and Lemma~\ref{lem:product} give, for example,
\[
\|uv_xu_{xx}\|_{H^{s-2}}
\le C_s\|uv_x\|_{H^{s-1}}\|u_{xx}\|_{H^{s-2}}
\le C_s\|u\|_{H^s}^2\|v\|_{H^s}.
\]
Similar estimates for the remaining terms show that
$F$ and $G$ are locally Lipschitz from
$H^{s-1}\times(H^s)^2$ into $H^s$ and satisfy
\begin{equation}\label{F-energy}
\begin{aligned}
\|F\|_{H^s}+\|G\|_{H^s}
&\le C_s\bigl[
\|u\|_{H^s}\|v\|_{H^s}
+(1+\|\eta\|_{H^{s-1}})^2
\bigr]\\
&\quad\times
\bigl(\|u\|_{H^s}+\|v\|_{H^s}\bigr).
\end{aligned}
\end{equation}

\proofstep{Step 1. The Lagrangian formulation.}
Let $X=\mathrm{id}+\xi$, with $\xi$ in a sufficiently
small neighborhood of zero in $H^s$, so that
$\inf_a X_a>0$.
Define
\begin{equation}\label{LX}
\partial_X=X_a^{-1}\partial_a,\qquad
L_X=1-\partial_X^2
=1-X_a^{-2}\partial_a^2+X_a^{-3}X_{aa}\partial_a.
\end{equation}
The product estimates give
\[
\|(L_X-L_{\mathrm{id}})h\|_{H^{s-2}}
\le C_s\|\xi\|_{H^s}\|h\|_{H^s}.
\]
Moreover, the map $\xi\mapsto L_X$ is smooth
in the operator norm from $H^s$ to $H^{s-2}$.
Since $L_{\mathrm{id}}=1-\partial_a^2$ is an
isomorphism between these spaces, a Neumann series
shows that $L_X^{-1}:H^{s-2}\to H^s$ exists and
depends smoothly on $\xi$ near zero.

Set $U=u\circ X$ and $V=v\circ X$.
The characteristic equation and the density equation give
\begin{equation}\label{flow}
X_t=UV,\qquad X(0,a)=a,\qquad
\rho\circ X=\frac{\bar\rho+\eta_0}{X_a}.
\end{equation}
In particular,
\[
\rho\circ X-\bar\rho
=\frac{\eta_0-\bar\rho\,\xi_a}{X_a}\in H^{s-1}.
\]
Define
\begin{equation}\label{lag-nonlocal}
\begin{aligned}
F_X=L_X^{-1}\Bigl[
&3UV\partial_XU
+\partial_X(U\partial_XU\partial_XV)
+(\partial_XU)^2\partial_XV\\
&+U\partial_XV\partial_X^2U
+\left(\frac{\bar\rho+\eta_0}{X_a}\right)^2U
\Bigr],
\end{aligned}
\end{equation}
and obtain $G_X$ by interchanging $U$ and $V$
and reversing the sign of the density term.
The expression in brackets belongs to $H^{s-2}$
and depends smoothly on $(\xi,U,V,\eta_0)$ in
$(H^s)^3\times H^{s-1}$.
Thus the transformed system
\begin{equation}\label{lag-ode}
\frac{d}{dt}(\xi,U,V)=(UV,-F_X,-G_X),\qquad
(\xi,U,V)|_{t=0}=(0,u_0,v_0)
\end{equation}
has a locally Lipschitz vector field in $(H^s)^3$,
with locally Lipschitz dependence on
$\eta_0\in H^{s-1}$.
\proofstep{Step 2. Strong solutions and continuous dependence.}
The Banach-space ODE theorem \cite{r6} gives a unique
local solution of \eqref{lag-ode} in
$C^1([0,T];(H^s)^3)$.
For initial data in a bounded set, $T>0$ can be
chosen uniformly so that $X_a$ remains bounded
away from zero.
Define
\begin{equation}\label{reconstruction}
u=U\circ X^{-1},\qquad
v=V\circ X^{-1},\qquad
\eta=\left(\frac{\eta_0-\bar\rho\,\xi_a}{X_a}\right)
\circ X^{-1}.
\end{equation}
By continuity of composition and inversion for $H^s$
diffeomorphisms \cite[Section 2 and Appendix B]{r20},
we have
$(\eta,u,v)\in C([0,T];H^{s-1}\times(H^s)^2)$.
The change-of-variable identities hold by approximation,
so these functions satisfy \eqref{3dp}.
The equations then give
$(\eta,u,v)\in
C^1([0,T];H^{s-2}\times(H^{s-1})^2)$.

Uniqueness follows by writing any two strong solutions
with the same initial data in Lagrangian coordinates
and applying uniqueness for \eqref{lag-ode}.
Continuous dependence follows from the ODE theorem
and continuity of the reconstruction map
\eqref{reconstruction}.
The equations then give continuous dependence
of the time derivatives.
\proofstep{Step 3. Lifespan and continuation.}
Set
\[
Y(t)=1+\|\eta(t)\|_{H^{s-1}}
+\|u(t)\|_{H^s}+\|v(t)\|_{H^s}.
\]
Applying the transport estimates and
Lemma~\ref{lem:commutator}, together with
\eqref{F-energy}, gives
\begin{equation}\label{Y-energy}
Y(t)\le Y_0+C_s\int_0^t Y(\tau)^3\,d\tau.
\end{equation}
At the stated regularity, this is justified by spatial
mollification and Lemma~\ref{lem:friedrichs}.
Comparison with $z'=C_sz^3$, $z(0)=Y_0$, gives
\[
Y(t)\le\frac{Y_0}{\sqrt{1-2C_sY_0^2t}}
\le\sqrt{2}\,Y_0
\]
as long as the solution exists and
$0\le t\le(4C_sY_0^2)^{-1}$.
The uniform local existence time allows continuation
as long as $Y$ remains bounded.
This proves the stated lifespan bound and the
continuation alternative.

On $\mathbb T$, the same argument uses the periodic
Green operator and lifts of periodic diffeomorphisms.
\end{proof}
\subsection{Continuation criteria}
\begin{theorem}[Continuation criterion]\label{thm:continuation}
Let $\bar\rho\in[0,1]$ and
$(\rho_0-\bar\rho,u_0,v_0)\in B^1_{2,1}\times(B^2_{2,1})^2$.
Let $(\rho,u,v)$ be the corresponding maximal solution
of \eqref{3dp} in this class, with lifespan $T^*$,
and let $X$ be its characteristic flow.
For $k_m=uv_x-5vu_x$ and $k_n=vu_x-5uv_x$, define
\[
I_T=\max_{j\in\{m,n\}}
\sup_{\substack{a\in\mathbb R\\0\le s\le t<T}}
\int_s^t k_j(\tau,X(\tau,a))\,d\tau,
\qquad 0<T\le T^*.
\]
If $T^*<\infty$, then $I_{T^*}=\infty$.
\end{theorem}

\begin{remark}\label{rem:continuation}
The same criterion holds in
$H^{s-1}\times(H^s)^2$ for $s\ge2$.
For both systems, setting
\[
Q(t)=\max\{\|(k_m)_+\|_{L^\infty},
\|(k_n)_+\|_{L^\infty}\}
=\bigl\|[-2(uv)_x+3|vu_x-uv_x|]_+\bigr\|_{L^\infty},
\]
with $f_+=\max\{f,0\}$, gives
\[
 I_T\le\int_0^T Q(t)\,dt,
\qquad
T^*<\infty\ \Longrightarrow\
\int_0^{T^*}Q(t)\,dt=\infty.
\]
\end{remark}
\begin{proof}[Proof of Theorem~\ref{thm:continuation}
and Remark~\ref{rem:continuation}]
Suppose that $T^*<\infty$ and $I_{T^*}<\infty$.
Since $k_m+k_n=-4(uv)_x$, we have
\[
X_a(t,a)
=\exp\left(\int_0^t(uv)_x(\tau,X(\tau,a))\,d\tau\right)
\ge e^{-I_{T^*}/2}.
\]
The identity $\rho(t,X)X_a=\rho_0$ gives
$\|\rho(t)\|_{L^\infty}
\le e^{I_{T^*}/2}\|\rho_0\|_{L^\infty}$.
The momentum equations imply
\[
\begin{aligned}
\frac{d}{dt}\bigl(m(t,X)\sqrt{X_a}\bigr)
&=\left(\frac12 k_m m-\rho^2u\right)(t,X)\sqrt{X_a},\\
\frac{d}{dt}\bigl(n(t,X)\sqrt{X_a}\bigr)
&=\left(\frac12 k_n n+\rho^2v\right)(t,X)\sqrt{X_a}.
\end{aligned}
\]
It follows that
\[
\|m(t)\|_{L^2}+\|n(t)\|_{L^2}
\le e^{I_{T^*}/2}
\bigl(\|m_0\|_{L^2}+\|n_0\|_{L^2}\bigr)
+C\int_0^t
\bigl(\|m(\tau)\|_{L^2}+\|n(\tau)\|_{L^2}\bigr)\,d\tau,
\]
where $C=C(I_{T^*},\|\rho_0\|_{L^\infty})$,
and we used the change of variables $x=X(t,a)$,
the definition of $I_{T^*}$, and
$\|D^{-2}f\|_{L^2}\le\|f\|_{L^2}$.
Gronwall's inequality therefore gives
\[
\sup_{t<T^*}
\bigl(\|u(t)\|_{H^2}+\|v(t)\|_{H^2}\bigr)<\infty.
\]

The density equation and its derivative with respect to $x$
can be written as
\[
(\partial_t+uv\partial_x)\eta=-\rho(uv)_x,\qquad
(\partial_t+uv\partial_x)\eta_x
=-2(uv)_x\eta_x-\rho(uv)_{xx}.
\]
Hence,
\begin{align*}
    \|\eta(t)\|_{H^1}
\le\|\eta_0\|_{H^1}
+C\int_0^t
\bigl(\|(uv)_x\|_{L^\infty}\|\eta\|_{H^1}
+\|\rho\|_{L^\infty}\|uv\|_{H^2}\bigr)(\tau)\,d\tau.
\end{align*}
Combining the preceding estimates and applying Gronwall's
inequality, we obtain
\[
\sup_{t<T^*}
\bigl(\|\eta(t)\|_{H^1}
+\|u(t)\|_{H^2}+\|v(t)\|_{H^2}\bigr)<\infty.
\]
We next estimate the Besov norms.
Since $vu_x$, $uv_x$, and
$\rho^2-\bar\rho^2=2\bar\rho\eta+\eta^2$
are uniformly bounded in $H^1$,
the product and transport estimates of
Section~\ref{sec:prelim} give
\[
\|m(t)\|_{B^0_{2,1}}+\|n(t)\|_{B^0_{2,1}}
\le e^{Ct}
\bigl(\|m_0\|_{B^0_{2,1}}+\|n_0\|_{B^0_{2,1}}\bigr),
\]
where $C$ depends on the preceding uniform bounds.
Similarly, the density equation gives
\[
\|\eta(t)\|_{B^1_{2,1}}
\le\|\eta_0\|_{B^1_{2,1}}
+C\int_0^t
\bigl(\|\eta\|_{B^1_{2,1}}
+\|u\|_{B^2_{2,1}}+\|v\|_{B^2_{2,1}}\bigr)(\tau)\,d\tau.
\]
Combining these estimates with the isomorphism
$1-\partial_x^2:B^2_{2,1}\to B^0_{2,1}$
and applying Gronwall's inequality, we obtain
\[
\sup_{t<T^*}
\bigl(\|\eta(t)\|_{B^1_{2,1}}
+\|u(t)\|_{B^2_{2,1}}+\|v(t)\|_{B^2_{2,1}}\bigr)<\infty.
\]
By Theorem~\ref{thm:local}, the solution extends beyond
$T^*$ in $H^1\times(H^2)^2$.
The preceding estimates, applied to smooth approximations,
give an extension in
$C([0,T^*+\delta];B^1_{2,1}\times(B^2_{2,1})^2)$
for some $\delta>0$, contradicting the maximality of $T^*$..

For $s\ge2$, the same argument, using the tame product
and transport estimates in $H^{s-2}$ for $m,n$
and in $H^{s-1}$ for $\eta$, proves the Sobolev
continuation criterion and persistence of higher regularity.

Finally, $k_m,k_n\le Q$ pointwise, so $ I_T\le\int_0^T Q(t)\,dt.$ The integral divergence in
Remark~\ref{rem:continuation} follows.
\end{proof}

\section{Blow-up}\label{sec:blowup}
We first prove Theorem~\ref{thm:blowup} and then obtain
a sharper blow-up condition from the coupled estimates
for $A$, $B$, and $N$.
We work with smooth 3DP solutions and extend the argument
to $s>3/2$ by approximation.
The GX case follows by taking $\rho_0=\bar\rho=0$.
\begin{proof}[Proof of Theorem~\ref{thm:blowup}]
As in \cite{r15,r17}, we estimate the slope along
a characteristic using a Riccati inequality.
The convolution estimates control the nonlocal terms
and ensure that $v$ retains its initial sign along
this characteristic for the required time.
We first prove the result for smooth solutions of 3DP system
and then extend it to $s>3/2$ by approximation.
\proofstep{Step 1. Closed convolution estimates.}
By the definitions of $A$ and $B$,
\[
\|u\|_{L^\infty}\le A,\qquad
\max\{\|v\|_{L^\infty},\|v_x\|_{L^\infty}\}\le B.
\]
If $f_t+(uvf)_x=R$, the absolute-value chain rule
and $|P_x|\le P$ give
\begin{equation}\label{convolution-rule}
\partial_t(P*|f|)
\le P*(|uv|\,|f|)+P*|R|.
\end{equation}
Moreover,
\[
\|P*|2uvu_x+u(g+\bar\rho^2)|\|_{L^\infty}
+2\|P*|uv_xu_x|\|_{L^\infty}
\le4BA^2+\bar\rho^2A+AN.
\]
Applying \eqref{convolution-rule} to the conservative
equations for $u$, $u_x$, $n$, and $g$, and using
$|P_x|\le P$ and $\|P\|_{L^1}=1$, we obtain
\begin{equation}\label{ABN}
\begin{aligned}
A(t)&\le A_0+\int_0^t
(12BA^2+2\bar\rho^2A+2AN)(\tau)\,d\tau,\\
B(t)&\le B_0+\int_0^t
(3AB^2+\bar\rho^2B+BN)(\tau)\,d\tau,\\
N(t)&\le N_0+\int_0^t
(4A^2B^2+4\bar\rho^2AB+3ABN)(\tau)\,d\tau.
\end{aligned}
\end{equation}
Integration by parts in the nonlocal terms of the
$v$ equation also gives
\begin{equation}\label{v-characteristic}
\begin{aligned}
(\partial_t+uv\partial_x)v
&=-P*\bigl\{
3uvv_x+(v^2+v_x^2)u_x-v(g+\bar\rho^2)
\bigr\}-P_x*(vu_xv_x).
\end{aligned}
\end{equation}

Set $S(t)=\bar\rho^2+A(t)^2+B(t)^2+N(t)$.
For smooth solutions, $A$, $B$, and $N$ are locally
absolutely continuous.
The differential estimates used to derive \eqref{ABN} give
\[
S'(t)\le10S(t)^2
\quad\text{for }t\in[0,T^*).
\]
Since $S(0)=K^2$ and $K\ge B_0\ge\nu>0$,
integrating this inequality yields, for any $k\ge K$,
\begin{equation}\label{S-bound}
S(t)\le\frac{k^2}{1-10k^2t},
\qquad
0\le t<\min\{T^*,(10k^2)^{-1}\}.
\end{equation}
In Steps~2--3, we use the coefficients
$S_\theta$, $T_\theta$, and $d_\theta$
defined in \eqref{threshold-coeff}, with $K$ replaced by $k$.

We also have $|uv_x|\le AB\le S/2$.
The nonlocal term in the slope equation satisfies
\[
\|P*|2uvu_x+u(g+\bar\rho^2)|\|_{L^\infty}
+2\|P*|uv_xu_x|\|_{L^\infty}
\le A(4AB+\bar\rho^2+N)\le2S^{3/2}.
\]
Similarly, \eqref{v-characteristic} gives
\[
\|(\partial_t+uv\partial_x)v\|_{L^\infty}
\le B(6AB+\bar\rho^2+N)\le3S^{3/2}.
\]
\proofstep{Step 2. A lower bound for $|v|$ along characteristics.}
Integrating the preceding estimate along $X(t,x_0)$
and using \eqref{S-bound}, we obtain
\begin{equation}\label{v-sign}
\sgn(v_0(x_0))v(t,X(t,x_0))
\ge\nu-\frac{3k}{5}
\bigl[(1-10k^2t)^{-1/2}-1\bigr].
\end{equation}
The definition of $T_\theta$ therefore gives
\[
\sgn(v_0(x_0))v(t,X(t,x_0))\ge\theta\nu,
\qquad S(t)\le S_\theta,
\qquad 0\le t\le T_\theta,\quad t<T^*.
\]
\proofstep{Step 3. The Riccati inequality and optimization.}
Set $y(t)=-\sgn(v_0(x_0))u_x(t,X(t,x_0))$.
The slope equation and the preceding estimates give
\[
y'\ge\theta\nu y^2-\frac12S_\theta y-2S_\theta^{3/2},
\qquad y\ge0.
\]
If $y_0\ge\Phi_\theta(k,\nu)$, the differential inequality
implies that $y(t)\ge y_0$ for
$0\le t<\min\{T^*,T_\theta\}$.
Integrating it and using the definition of $\Phi_\theta$,
we obtain
\[
T^*\le\tau_\theta(k,\nu,y_0)\le T_\theta.
\]
Indeed, if $T^*>\tau_\theta(k,\nu,y_0)$,
the same inequality forces $y(t)\to\infty$ as
$t\uparrow\tau_\theta(k,\nu,y_0)$,
contradicting the boundedness of $\|u_x(t)\|_{L^\infty}$
on compact subintervals of $[0,T^*)$.
If $y_0>\Phi_\theta(k,\nu)$, the explicit formula gives
$\tau_\theta(k,\nu,y_0)<T_\theta$.

The threshold can be optimized over $\theta\in(0,1)$.
Indeed, $\theta\mapsto\Phi_\theta(k,\nu)$ is continuous and
\[
\lim_{\theta\downarrow0}\Phi_\theta(k,\nu)
=\lim_{\theta\uparrow1}\Phi_\theta(k,\nu)=\infty,
\]
so there exists $\theta_*\in(0,1)$ such that
$\Phi_{\theta_*}(k,\nu)\le\Phi_\theta(k,\nu)$
for all $\theta\in(0,1)$.
The scaling $(u,v,\rho)\mapsto(\lambda u,\lambda^{-1}v,\rho)$,
$\lambda>0$, also preserves the system and the lifespan,
while
\[
(A_0,B_0,N_0,\nu,y_0)
\longmapsto
(\lambda A_0,\lambda^{-1}B_0,N_0,
\lambda^{-1}\nu,\lambda y_0).
\]
Applying the same threshold to the scaled data
and taking the infimum over $\lambda>0$ and
$\theta\in(0,1)$ gives \eqref{scaling-criterion}.
\proofstep{Step 4. Passage to low regularity.}
Fix $T<T^*$ and mollify the initial data with a
nonnegative kernel $\kappa_\epsilon$ of unit mass,
supported in $[-\epsilon,\epsilon]$.
By Theorem~\ref{thm:local} and Sobolev embedding,
\[
(\rho_\epsilon-\rho,u_\epsilon-u,v_\epsilon-v)
\to0
\quad\text{in }C([0,T];L^\infty\times(W^{1,\infty})^2).
\]
It follows from \eqref{flow} that
$X_\epsilon\to X$ and $X_{\epsilon,a}\to X_a$
uniformly on $[0,T]\times\mathbb R$, with
\[
C_T^{-1}\le X_{\epsilon,a}\le C_T.
\]
The characteristic formula for $n_\epsilon$ then gives
\[
\sup_{\epsilon,\,t\le T}
\|n_\epsilon(t)\|_{L^1_{\rm uloc}}
\le C_T(1+\|n_0\|_{L^1_{\rm uloc}}).
\]
Since $n_{0,\epsilon}\to n_0$ in $L^1_{\rm loc}$,
the characteristic formula and the uniform Jacobian
bounds give
\[
n_\epsilon\to n
\quad\text{in }C([0,T];L^1_{\rm loc}),
\]
where $n=v-v_{xx}$ in distributions.
This follows from local approximation of $n_0$
by compactly supported continuous functions.
The identity $g_\epsilon=n_\epsilon u_{\epsilon,x}
+\rho_\epsilon^2-\bar\rho^2$ then gives
$g_\epsilon\to g$ in $C([0,T];L^1_{\rm loc})$.
Together with the uniform $L^1_{\rm uloc}$ bounds
and the exponential decay of $P$, this convergence
allows passage to the limit in the convolution estimates.

The initial quantities satisfy
\[
A_{0,\epsilon}\le A_0,\qquad
B_{0,\epsilon}\le B_0,\qquad
N_{0,\epsilon}
\le N_0+\omega_{u_{0,x}}(2\epsilon)B_0
+\omega_{\rho_0}(2\epsilon)^2=N_0+o(1),
\]
where $\omega_f$ denotes the uniform modulus of
continuity of $f$.
The error terms vanish since $s>3/2$.
Continuous dependence for the ODE bounds therefore
gives the estimates of Step~1 with initial values
$A_0$, $B_0$, and $N_0$.
Passing to the limit along characteristics gives
the estimates of Steps~2--3.
The Riccati argument applies directly to the limit
solution, including the case $y_0=\Phi_\theta(K,\nu)$,
so no blow-up condition is required for the mollified data.
The same argument applies to the coupled criterion below.
\end{proof}

\Needspace{7\baselineskip}
\subsection*{A coupled blow-up criterion}
\begingroup
\begin{proposition}\label{prop:coupled}
Assume the regularity and momentum conditions of
Theorem~\ref{thm:blowup}, with $A_0>0$,
and choose $x_0$ such that $v_0(x_0)\ne0$.
Define $\nu$ and $y_0$ as in that theorem.
Let $(A_c,B_c,N_c)$ solve the system obtained by
replacing the inequalities in \eqref{ABN} with equalities,
with initial values $(A_0,B_0,N_0)$.
Let $T_c>0$ satisfy
\begin{equation}\label{H-coupled}
\int_0^{T_c}B_c(s)
\bigl(6A_c(s)B_c(s)+\bar\rho^2+N_c(s)\bigr)\,ds=\nu,
\end{equation}
and let $r$ solve
\begin{equation}\label{terminal}
\begin{aligned}
r'
&=-\nu+\int_0^t B_c(s)
\bigl(6A_c(s)B_c(s)+\bar\rho^2+N_c(s)\bigr)\,ds\\
&\quad+A_cB_c\,r
+A_c(4A_cB_c+\bar\rho^2+N_c)r^2,
\qquad r(T_c)=0.
\end{aligned}
\end{equation}
Then $T_c<\infty$ and $r(0)>0$.
Setting $y_c=r(0)^{-1}$, we have
\[
y_0\ge y_c\ \Longrightarrow\ T^*\le T_c,
\qquad
y_0>y_c\ \Longrightarrow\ T^*<T_c.
\]
Moreover,
\[
y_c<\inf_{0<\theta<1}\Phi_\theta(K,\nu),
\]
and the blow-up condition is invariant under
$(u,v,\rho)\mapsto(\lambda u,\lambda^{-1}v,\rho)$,
$\lambda>0$.
\end{proposition}

\begin{proof}
The comparison principle applied to \eqref{ABN} gives
$A\le A_c$, $B\le B_c$ and $N\le N_c$.
The functions $A_c$, $B_c$ and $N_c$ are nondecreasing,
and
\[
B_c'\le B_c(6A_cB_c+\bar\rho^2+N_c),\qquad
(\log A_c)'
\le\frac{2}{B_0}B_c(6A_cB_c+\bar\rho^2+N_c).
\]
As long as the integral in \eqref{H-coupled}
does not exceed $\nu$, integration gives
\[
B_c(t)\le B_0+\nu,\qquad
A_c(t)\le A_0e^{2\nu/B_0}.
\]
The equation for $N_c$ and Gronwall's inequality
then bound $N_c$ on every finite time interval.
Thus the comparison solution cannot cease to exist
before the integral reaches $\nu$.
Since the integrand is at least $6A_0B_0^2>0$, we then have $ 0<T_c\le\frac{\nu}{6A_0B_0^2}<\infty$ and the comparison solution extends beyond $T_c$.

For $y>0$, the estimates in Steps~1--2 show that
$1/y$ satisfies the differential inequality obtained
from \eqref{terminal} by replacing $=$ with $\le$.
After reversing time, the forcing term is nonnegative
and the linear and quadratic coefficients are nonpositive.
Thus $r$ exists on $[0,T_c]$ and is positive for $t<T_c$.

If $y_0\ge y_c$, the differential inequality gives
$0<1/y(t)\le r(t)$ for $0\le t<\min\{T^*,T_c\}$.
Since $r(T_c)=0$, we obtain $T^*\le T_c$.
For $y_0>y_c$, subtracting the inequality for $1/y$
from \eqref{terminal} yields
\[
r(t)-\frac1{y(t)}
\ge r(0)-\frac1{y_0}>0.
\]
As $r(t)\to0$ when $t\uparrow T_c$, the solution
cannot exist on all of $[0,T_c)$.
Hence $T^*<T_c$.
The approximation argument in Step~4 extends these
conclusions to $s>3/2$..

To compare the two blow-up conditions, set $k=K$ and let
\[
T_{\rm sc}
=\frac{1-(1+5\nu/(3k))^{-2}}{10k^2},
\]
the time at which the lower bound in \eqref{v-sign}
vanishes.
Applying the estimate used in \eqref{S-bound}
to the comparison system gives
\[
\bar\rho^2+A_c^2+B_c^2+N_c
\le\frac{k^2}{1-10k^2t}.
\]
Since $A_c,B_c>0$, we have
\[
\begin{aligned}
A_cB_c
\le\frac{k^2}{2(1-10k^2t)},\\
B_c(6A_cB_c+\bar\rho^2+N_c)
&<3\left(\frac{k^2}{1-10k^2t}\right)^{3/2},\\
A_c(4A_cB_c+\bar\rho^2+N_c)
&<2\left(\frac{k^2}{1-10k^2t}\right)^{3/2}.
\end{aligned}
\]
Integrating the first inequality gives $T_c>T_{\rm sc}$.
Let $r_{\rm sc}$ solve the terminal problem obtained
from \eqref{terminal} by replacing the three coefficient
expressions above with their respective upper bounds
and imposing $r_{\rm sc}(T_{\rm sc})=0$.
Since $r(T_{\rm sc})>0$, comparison backward in time yields
$r(0)>r_{\rm sc}(0)$.

For $\theta\in(0,1)$, we have $T_\theta<T_{\rm sc}$
and $r_{\rm sc}(T_\theta)>0$.
On $[0,T_\theta]$, the velocity lower bound is at least
$\theta\nu$, while the linear and quadratic coefficients
are at most $S_\theta/2$ and $2S_\theta^{3/2}$.
Comparing backward with the reciprocal of the
constant-coefficient solution in Step~3 gives
\[
y_c<1/r_{\rm sc}(0)<\Phi_\theta(k,\nu).
\]
Taking $\theta=\theta_*$ as chosen in Step~3 yields
$y_c<\inf_{0<\theta<1}\Phi_\theta(k,\nu)$.

Under reciprocal scaling,
$(A_c,B_c,N_c,r)\mapsto
(\lambda A_c,\lambda^{-1}B_c,N_c,\lambda^{-1}r)$,
while $T_c$ is unchanged.
Hence $y_c\mapsto\lambda y_c$ and
$y_0\mapsto\lambda y_0$, proving the invariance
of the blow-up condition.
\end{proof}

\begin{remark}
The comparison system can be reduced to a linear system
by the change of variable $d\tau/dt=A_cB_c$, $\tau(0)=0$:
\[
\frac{d}{d\tau}
\binom{A_cB_c}{\bar\rho^2+N_c}
=
\begin{pmatrix}
15&3\\
4&3
\end{pmatrix}
\binom{A_cB_c}{\bar\rho^2+N_c}
+\binom{0}{\bar\rho^2},
\]
with initial value $(A_0B_0,\bar\rho^2+N_0)^\top$.
The matrix has eigenvalues $9\pm4\sqrt3$, and
\[
\begin{aligned}
B_c(t(\tau))
&=B_0
\left(\frac{A_c(t(\tau))B_c(t(\tau))}{A_0B_0}\right)^{1/3}
e^{-2\tau},\\
t(\tau)
&=\int_0^\tau
\frac{d\xi}{A_c(t(\xi))B_c(t(\xi))}.
\end{aligned}
\]
The equation for $B_c$ also gives
\[
\begin{aligned}
&\nu-\int_0^{t(\tau)}
B_c(s)\bigl(6A_c(s)B_c(s)+\bar\rho^2+N_c(s)\bigr)\,ds\\
&\qquad=\nu+B_0-B_c(t(\tau))
-3\int_0^\tau B_c(t(\xi))\,d\xi.
\end{aligned}
\]
If $\tau_c$ is the first zero of this expression,
then $T_c=t(\tau_c)$, and \eqref{terminal} gives
$y_c=1/r(0)$.
\end{remark}
\begin{remark}
For the $v_x$ equation, the corresponding mixed quantity is $mv_x-(\rho^2-\bar\rho^2)$. The minus sign comes from the opposite signs of the density terms in the two momentum equations.
\end{remark}

\section{Critical norm inflation}\label{sec:inflation}

We construct perturbations with small critical norms
and large slopes, while keeping the initial convolution
quantities uniformly bounded.
The blow-up result gives short lifespans, and the following
criterion is used to prove norm inflation in $u$.

\begin{lemma*}[One-component blow-up criterion]
Let $s\ge2$, $\bar\rho\in[0,1]$, and
$(\rho_0-\bar\rho,u_0,v_0)\in H^{s-1}\times(H^s)^2$.
Let $(\rho,u,v)$ be the corresponding maximal strong
solution of \eqref{3dp}, with lifespan $T^*$.
If $T^*<\infty$ and
\[
\sup_{t<T^*}
\bigl(\|u(t)\|_{L^\infty}
+\|v(t)\|_{W^{1,\infty}}\bigr)<\infty,
\]
then
\[
\int_0^{T^*}\|u(t)\|_{B^1_{\infty,\infty}}\,dt=\infty.
\]
The same conclusion holds for GX system with
$(u_0,v_0)\in(H^s)^2$.
\end{lemma*}
\begin{proof}
\begingroup
Suppose that the stated time integral is finite.
We combine a local momentum estimate with a logarithmic
estimate to bound $\|u_x\|_{L^\infty}$.
Set
\[
Z(t)=e+\|P*|m(t)|^2\|_{L^\infty}^{1/2}
+\|\rho(t)\|_{L^\infty}^2,
\]
where the last term vanishes for GX system.
The momentum equation gives
\[
(m^2)_t+(uvm^2)_x
=(uv_x-5vu_x)m^2-2\rho^2um.
\]
Convolving with $P$ and using $|P_x|\le P$,
the weighted Cauchy--Schwarz inequality, and the
density equation along characteristics, we obtain
\[
Z(t)\le Z(0)+C_{T^*}\int_0^t
\bigl(1+\|u_x(\tau)\|_{L^\infty}\bigr)Z(\tau)\,d\tau,
\]
where $C_{T^*}$ depends only on $T^*$ and the assumed
bound for $\|u\|_{L^\infty}+\|v\|_{W^{1,\infty}}$.
Gronwall's inequality then yields
\[
\log Z(t)\le\log Z(0)+C_{T^*}\int_0^t
\bigl(1+\|u_x(\tau)\|_{L^\infty}\bigr)\,d\tau.
\]

Since $m=(1-\partial_x^2)u$ and
$\|m\|_{L^2_{\rm uloc}}
\le C\|P*|m|^2\|_{L^\infty}^{1/2}$,
a dyadic decomposition gives, for $J\ge0$,
\[
\|u_x\|_{L^\infty}
\le C_{T^*}
+C(J+1)\|u\|_{B^1_{\infty,\infty}}
+C2^{-J/2}Z.
\]
Indeed, the lowest-frequency block is bounded by
$C\|u\|_{L^\infty}$, and the blocks $0\le j\le J$
contribute at most
$C(J+1)\|u\|_{B^1_{\infty,\infty}}$.
For $j>J$, the frequency-localized kernel of
$\partial_x(1-\partial_x^2)^{-1}$
decays rapidly and has $L^2$ norm bounded by $C2^{-j/2}$.
Splitting the convolution into unit intervals gives
\[
\|\Delta_j u_x\|_{L^\infty}
\le C2^{-j/2}\|m\|_{L^2_{\rm uloc}},
\]
and summing over $j>J$ yields the last term.

Choosing $J=\lceil2\log_2 Z\rceil$ and substituting
into the estimate for $\log Z$, we obtain
\[
\log Z(t)\le\log Z(0)+C_{T^*}t
+C_{T^*}\int_0^t
\|u(\tau)\|_{B^1_{\infty,\infty}}
\log Z(\tau)\,d\tau.
\]
The assumed integrability and Gronwall's inequality
therefore imply $\sup_{t<T^*}Z(t)<\infty$.
Since $u_x=P_x*m$, $|P_x|\le P$, and $\|P\|_{L^1}=1$,
Cauchy--Schwarz gives
\[
\|u_x(t)\|_{L^\infty}
\le\|P*|m(t)|^2\|_{L^\infty}^{1/2}
\le Z(t).
\]
Thus $u,v$ are uniformly bounded in $W^{1,\infty}$
on $[0,T^*)$, contradicting
Remark~\ref{rem:continuation}.

The above estimates are justified for $s\ge2$
by a standard approximation argument.
\endgroup
\end{proof}

\begin{proof}[Proof of Theorem~\ref{thm:inflation}]
Following \cite{r21,r24}, we construct perturbations
with small critical norms and large slopes,
while keeping the initial convolution quantities
uniformly bounded.
The blow-up result and the preceding lemma then
yield norm inflation in $u$.
\proofstep{Step 1. Construction of the initial data.}
Fix $0<\varepsilon<1$ and choose
$(\tilde\eta,\tilde u,\tilde v)\in\mathcal S(\mathbb R)^3$
such that
\[
\|\tilde\eta-\eta_*\|_{B^{1/p}_{p,r}}
+\|\tilde u-u_*\|_{B^{1+1/p}_{p,r}}
+\|\tilde v-v_*\|_{B^{1+1/p}_{p,r}}
<\frac{\varepsilon}{2}.
\]
Since $\tilde v$ tends to zero at infinity, a sufficiently
small Schwartz perturbation preserves this inequality
and ensures that $\tilde v(x_0)=\nu>0$ for some
$x_0\in\mathbb R$.
Choose a real-valued $\phi\in\mathcal S(\mathbb R)$
with Fourier support in an annulus away from zero
and $\phi'(0)<0$, and positive coefficients
$(a_j)\in\ell^r$ with $\sum_{j\ge1}a_j=\infty$.
Set
\[
\rho_{0,M}=\bar\rho+\tilde\eta,\qquad
u_{0,M}=\tilde u+\delta\sum_{j=1}^M
a_j2^{-j}\phi\bigl(2^j(x-x_0)\bigr),\qquad
v_{0,M}=\tilde v.
\]

Frequency localization and scaling give
\[
\begin{gathered}
\|u_{0,M}-\tilde u\|_{B^{1+1/p}_{p,r}}
\le C\delta\|(a_j)\|_{\ell^r},
\qquad
\|u_{0,M}-\tilde u\|_{W^{1,1}}\le C\delta,\\
y_{0,M}=-u_{0,M,x}(x_0)
=-\tilde u'(x_0)
+\delta|\phi'(0)|\sum_{j=1}^M a_j
\longrightarrow\infty,
\end{gathered}
\]
where $C$ is independent of $M$.
Fix $\delta>0$ sufficiently small that the initial
distance from $(\eta_*,u_*,v_*)$ is less than
$\varepsilon$ for every $M$.

Since $n_{0,M}=\tilde v-\tilde v''$,
\[
\|g_{0,M}\|_{L^1}
\le\|\tilde v-\tilde v''\|_{L^\infty}
\bigl(\|\tilde u'\|_{L^1}+C\delta\bigr)
+\|2\bar\rho\tilde\eta+\tilde\eta^2\|_{L^1}.
\]
Together with the $W^{1,1}$ bound above, Young's inequality
gives $\sup_M(A_{0,M}+B_{0,M}+N_{0,M})<\infty$.
We may therefore choose $k>0$, independent of $M$, such that
\[
\bar\rho^2+A_{0,M}^2+B_{0,M}^2+N_{0,M}\le k^2.
\]
\proofstep{Step 2. Norm inflation in $u$.}
For sufficiently large $M$,
$T_M^*<T_{1/2}<(10k^2)^{-1}$.
Hence \eqref{S-bound} gives
\[
\sup_{t<T_M^*}
\bigl(\|u(t)\|_{L^\infty}
+\|v(t)\|_{W^{1,\infty}}\bigr)<\infty.
\]
The preceding lemma and the embedding
$B^{1+1/p}_{p,r}\hookrightarrow B^1_{\infty,\infty}$
therefore imply
\[
\int_0^{T_M^*}\|u(t)\|_{B^{1+1/p}_{p,r}}\,dt=\infty.
\]
Since the solution norm is bounded on compact lifespan
intervals and $u_*\in B^{1+1/p}_{p,r}$, we obtain
\[
\limsup_{t\uparrow T_M^*}
\|u(t)-u_*\|_{B^{1+1/p}_{p,r}}=\infty.
\]
Choose $0<t_\varepsilon<T_M^*<\varepsilon$ such that
$\|u(t_\varepsilon)-u_*\|_{B^{1+1/p}_{p,r}}
>\varepsilon^{-1}$.

For GX, take $\bar\rho=0$, background data $(0,u_*,v_*)$,
and $\tilde\eta=0$ in the construction.
Then $\rho\equiv0$, and the same argument applies.
\end{proof}

\section{Global well-posedness for GX}\label{sec:global}

\begin{proof}[Proof of Theorem~\ref{thm:global}]
The sign conditions and conservation of the mixed energy
give uniform bounds on the transport coefficients.
For $\sigma\ge0$, global existence follows from
the continuation criterion.
For $-1/2<\sigma<0$, we also use positivity of the
Bessel kernel to estimate the $H^\sigma$ norms.
We first establish these estimates for smooth solutions
and then pass to the limit by approximation.
\proofstep{Step 1. Signs and energy estimates.}
For a smooth solution, the characteristic formula gives
\[
m(t,X(t,a))
=m_0(a)\exp\left(
-3\int_0^t(vu_x)(\tau,X(\tau,a))\,d\tau
\right).
\]
The analogous formula for $n$ has $uv_x$ in place
of $vu_x$, so both momenta preserve their initial signs.
Since $P\ge0$ and $|P_x|\le P$, we have
\begin{equation}\label{positive-velocity}
|u|=P*(\varepsilon_1m),\qquad
|v|=P*(\varepsilon_2n),\qquad
|u_x|\le|u|,\quad |v_x|\le|v|.
\end{equation}

As in \cite[Lemma 5.1]{r22},
the mixed energy
$E(t)=\int_{\mathbb R}(uv+u_xv_x)\,dx$
is conserved.
Indeed,
\[
E'(t)=2\int_{\mathbb R}
(uv_x-vu_x)(vm-un)\,dx=0,
\qquad
vm-un=\partial_x(uv_x-vu_x).
\]
Writing $E_*=|E(0)|$, the sign conditions give
\[
E_*=\int_{\mathbb R}
(\varepsilon_1m)\,P*(\varepsilon_2n)\,dx
=\iint_{\mathbb R^2}
P(y-z)(\varepsilon_1m(y))(\varepsilon_2n(z))\,dy\,dz.
\]
Moreover,
\[
P(x-y)P(x-z)
=\frac14e^{-|x-y|-|x-z|}
\le\frac14e^{-|y-z|}
=\frac12P(y-z).
\]
Multiplying by $(\varepsilon_1m(y))(\varepsilon_2n(z))$
and integrating, we obtain
\begin{equation}\label{energy-coeff}
|uv|\le E_*/2,\qquad
|vu_x|,|uv_x|\le E_*/2,\qquad
\|(uv)_x\|_{L^\infty}\le E_*.
\end{equation}
The integrals are finite for smooth Sobolev data,
and the identities are justified by spatial cutoffs;
no global $L^1$ assumption on the momenta is needed.
In particular, $Q(t)\le3E_*$.
For $\sigma\ge0$, the continuation criterion in
Section~\ref{sec:local} therefore excludes a finite
maximal existence time.
\proofstep{Step 2. The $L^2$ case and higher regularity.}
Integration by parts and \eqref{energy-coeff} give
\[
\left|\frac{d}{dt}\|m\|_{L^2}^2\right|
=\left|\int_{\mathbb R}(uv_x-5vu_x)m^2\,dx\right|
\le3E_*\|m\|_{L^2}^2,
\]
and the same estimate holds for $n$.
Thus
\[
\|m(t)\|_{L^2}+\|n(t)\|_{L^2}
\le e^{3E_*|t|/2}
\bigl(\|m_0\|_{L^2}+\|n_0\|_{L^2}\bigr).
\]
The continuation alternative gives global existence
in $(L^2)^2$, while the tame estimates in
Remark~\ref{rem:continuation}, with $\rho=0$,
bound the higher Sobolev norms on finite time intervals.
Hence smooth initial data with the prescribed signs
give global smooth solutions.
Approximation preserving these signs extends the result
to $\sigma\ge0$.
\proofstep{Step 3. Negative Sobolev estimates.}
Set $\theta=-\sigma\in(0,1/2)$.
The kernel $K_\theta$ of $(1-\partial_x^2)^{-\theta}$
is nonnegative, since
\[
K_\theta(x)
=\frac1{\Gamma(\theta)}
\int_0^\infty
r^{\theta-1}e^{-r}(4\pi r)^{-1/2}e^{-x^2/(4r)}\,dr.
\]
For nonnegative $f\in L^2$ and bounded nonnegative $h$,
\begin{equation}\label{positive-weight}
\begin{aligned}
\|hf\|_{H^{-\theta}}^2
&=\iint_{\mathbb R^2}
K_\theta(x-y)h(x)f(x)h(y)f(y)\,dx\,dy\\
&\le\|h\|_{L^\infty}^2\|f\|_{H^{-\theta}}^2.
\end{aligned}
\end{equation}
The kernel representation and Tonelli's theorem
justify this identity without assuming $f\in L^1$.

Let $X$ be an increasing $C^1$ diffeomorphism
with $L^{-1}\le X'\le L$, where $L\ge1$.
A change of variables gives
\[
\|\varphi\circ X\|_{L^2}\le L^{1/2}\|\varphi\|_{L^2},
\qquad
\|\varphi\circ X\|_{H^1}\le L^{1/2}\|\varphi\|_{H^1}.
\]
By density and complex interpolation,
\[
\|\varphi\circ X\|_{H^\theta}
\le L^{1/2}\|\varphi\|_{H^\theta},
\qquad 0\le\theta\le1.
\]
For $f\in L^2$,
\[
\int_{\mathbb R}
\left(\frac f{X'}\right)\!\circ X^{-1}(x)\,
\varphi(x)\,dx
=\int_{\mathbb R}f(a)\varphi(X(a))\,da,
\]
so duality yields
\begin{equation}\label{pushforward-bound}
\left\|
\left(\frac f{X'}\right)\!\circ X^{-1}
\right\|_{H^{-\theta}}
\le L^{1/2}\|f\|_{H^{-\theta}}.
\end{equation}

For the smooth solutions constructed above,
\eqref{energy-coeff} gives
\[
e^{-E_*|t|}\le X_a(t,a)\le e^{E_*|t|}.
\]
The conservative momentum equation implies
\[
\varepsilon_1m(t)
=\left(
\frac{h_m(t)\varepsilon_1m_0}{X_a(t)}
\right)\!\circ X(t)^{-1},
\qquad
h_m(t,a)
=\exp\left(
\int_0^t(uv_x-2vu_x)(\tau,X(\tau,a))\,d\tau
\right).
\]
The corresponding formula for $\varepsilon_2n$
has weight $h_n$, obtained by replacing
$uv_x-2vu_x$ with $vu_x-2uv_x$.
By \eqref{energy-coeff},
$0<h_m,h_n\le e^{3E_*|t|/2}$.
Applying \eqref{positive-weight} and
\eqref{pushforward-bound} with $L=e^{E_*|t|}$, we obtain
\[
\|m(t)\|_{H^{-\theta}}
\le e^{E_*|t|/2}
\|h_m(t)\varepsilon_1m_0\|_{H^{-\theta}}
\le e^{2E_*|t|}\|m_0\|_{H^{-\theta}}.
\]
The same estimate for $n$ proves \eqref{negative-bound}.
\proofstep{Step 4. Approximation and continuation.}
For $-1/2<\sigma<0$, set
\[
m_0^\delta=e^{\delta\partial_x^2}m_0,\qquad
n_0^\delta=e^{\delta\partial_x^2}n_0.
\]
These approximations preserve the prescribed signs
and converge in $H^\sigma$.
The corresponding velocities converge in $H^1$,
so $E_*^\delta\to E_*$.
The estimates of Step~3 are therefore uniform
on every finite time interval.

By Theorem~\ref{thm:local}, the smooth solutions
converge on compact subintervals of the limiting
solution's lifespan.
Thus \eqref{negative-bound}, the sign conditions,
and conservation of the mixed energy pass to the limit.
The resulting norm bound and the local existence time
exclude a finite maximal existence time.
Uniqueness and time regularity follow from
Theorem~\ref{thm:local}.
For convergent initial data with the prescribed signs,
the same uniform bounds allow us to apply local
continuous dependence on successive intervals,
giving continuous dependence on every finite
forward time interval.

Finally, the transformation
$(u(t),v(t))\mapsto(-u(-t),v(-t))$ preserves GX.
Applying the forward result to $(-m_0,n_0)$ gives
the negative-time solution, with the estimates
expressed in terms of $|t|$.
Uniqueness joins the two solutions at $t=0$,
and continuous dependence on compact intervals
of $\mathbb R$ follows in the same way.
\end{proof}

\begin{remark}
For the solutions of Theorem~\ref{thm:global},
\[
e^{-E_*|t|}\|m_0\|_{H^{-1}}
\le\|m(t)\|_{H^{-1}}
\le e^{E_*|t|}\|m_0\|_{H^{-1}},
\qquad t\in\mathbb R,
\]
and the same estimate holds for $n$.
Indeed, for smooth data with the prescribed signs,
\[
\frac{d}{dt}\|u\|_{H^1}^2
=2\int_{\mathbb R}um(uv_x-vu_x)\,dx,
\qquad
\left|\frac{d}{dt}\|u\|_{H^1}^2\right|
\le2E_*\|u\|_{H^1}^2,
\]
since $um\ge0$ and $|uv_x|,|vu_x|\le E_*/2$.
Integration and the identity
$\|m\|_{H^{-1}}=\|u\|_{H^1}$ give the stated bounds.
Sign-preserving approximation extends them to
the data in Theorem~\ref{thm:global}.

A sign condition on only one momentum does not
ensure global existence.
Choose a nonzero $n_0\in C_c^\infty$ with $n_0\ge0$,
so that $v_0=P*n_0>0$.
Take $u_{0,M}=h_M$ from \eqref{frequency-data},
with fixed $x_0$, $\delta>0$, and $a_j=j^{-3/4}$.
The initial convolution quantities are uniformly bounded,
whereas $-u_{0,M}'(x_0)\to\infty$.
Theorem~\ref{thm:blowup} gives $T_M^*\to0$,
although $n_M(t)\ge0$ throughout the lifespan
by the characteristic formula.
\end{remark}

\begin{remark}
For 3DP, the density terms cancel in the mixed-energy
identity, but the momentum signs need not be preserved.
Indeed, if smooth initial data satisfy
$m_0(x_0)=0$, $u_0(x_0)>0$, and $\rho_0(x_0)\ne0$, then
\[
(\partial_t+uv\partial_x)m(0,x_0)
=-\rho_0(x_0)^2u_0(x_0)<0.
\]
Thus the preceding global existence argument does not
extend to general 3DP data with nonzero density.
\end{remark}

\section{Examples}\label{sec:examples}
The following examples show that the blow-up result allows
slowly decaying density perturbations and that norm inflation
does not require proportional background velocities.

\begin{example}[Blow-up with a slowly decaying density]
Let $\phi(x)=xe^{-x^2}$ and, for $0<\epsilon\le1$, set
\[
u_0^\epsilon(x)=-\epsilon^{1/2}\phi(x/\epsilon),\qquad
v_0(x)=e^{-x^2},\qquad
n_0(x)=(3-4x^2)e^{-x^2}.
\]
Take $\bar\rho=1$ and $\rho_0=1+f_a$, where
$f_a(x)=(1+x^2)^{-a/2}$ with $1/2<a<1$.
Then
\[
f_a\in\bigcap_{k\ge0}H^k,\qquad
\rho_0^2-1=2f_a+f_a^2\notin L^1,
\]
since $f_a\in L^2\setminus L^1$ and its derivatives
decay faster.
Moreover,
\[
N_0^\epsilon
\le\|n_0\|_{L^\infty}
\|P*|u_{0,x}^\epsilon|\|_{L^\infty}
+\|2f_a+f_a^2\|_{L^\infty}
\le\frac32\epsilon^{1/2}\|\phi'\|_{L^1}+3.
\]
Also, $A_0^\epsilon=O(\epsilon^{1/2})$ and $B_0$
is independent of $\epsilon$, whereas
\[
v_0(0)=1,\qquad
-u_{0,x}^\epsilon(0)=\epsilon^{-1/2}\to\infty.
\]
Choose $k>0$, independent of $\epsilon$, such that
$1+(A_0^\epsilon)^2+B_0^2+N_0^\epsilon\le k^2$.
For any fixed $\theta\in(0,1)$,
$\epsilon^{-1/2}\ge\Phi_\theta(k,1)$
when $\epsilon$ is sufficiently small.
Theorem~\ref{thm:blowup} therefore gives finite-time blow-up.
The reciprocal-slope estimate in
Section~\ref{sec:inflation} further yields
\[
T_\epsilon^*\le\epsilon^{1/2}+O(\epsilon).
\]
\end{example}
\begin{example}[Norm inflation near a nonproportional background]
Take $\bar\rho=1$ and
\[
(\eta_*,u_*,v_*)
=(e^{-x^2},e^{-x^2},e^{-2x^2})
\in H^{1/2}\times(H^{3/2})^2.
\]
Using $\phi$ from the construction in
Section~\ref{sec:inflation}, apply the
construction with $x_0=0$, $a_j=j^{-3/4}$, and
$M_\delta=\lceil\delta^{-8}\rceil$.
The resulting initial data satisfy
\[
\|\rho_{0,\delta}-1-\eta_*\|_{H^{1/2}}
+\|u_{0,\delta}-u_*\|_{H^{3/2}}
+\|v_{0,\delta}-v_*\|_{H^{3/2}}
\le C\delta,
\]
with
\[
v_{0,\delta}(0)=1,\qquad
y_{0,\delta}
=\delta|\phi'(0)|\sum_{j=1}^{M_\delta}j^{-3/4}
\asymp\delta^{-1}.
\]
The initial convolution quantities are uniformly bounded.
Thus, for sufficiently small $\delta>0$,
the lifespan estimate in Section~\ref{sec:inflation} gives
\[
T_\delta^*
\le\frac1{y_{0,\delta}}+O(y_{0,\delta}^{-2})
=O(\delta).
\]
The one-component criterion then yields
$0<t_\delta<T_\delta^*$ such that
\[
\|u_\delta(t_\delta)-u_*\|_{H^{3/2}}>\delta^{-1}.
\]
Hence the initial data converge in
$H^{1/2}\times(H^{3/2})^2$, while norm inflation
occurs in $u$ at times tending to zero.
Taking $\rho_{0,\delta}=\bar\rho=0$ gives the
corresponding GX example near
$(u_*,v_*)=(e^{-x^2},e^{-2x^2})$.
\end{example}

\begingroup

\section*{Declarations}\addcontentsline{toc}{section}{Declarations}
\noindent\textbf{Funding.} Z. Yin was supported by the National Natural Science Foundation of China (No. 12571261).\\
\textbf{Author contributions.} S. Liu and Z. Yin contributed equally to this work.\\
\textbf{Data availability statement.} Data sharing is not applicable to this article, as no datasets were generated or {analyzed during the study.}\\
\textbf{Conflict of interest.} The authors declare that they have no conflict of interest.


\addcontentsline{toc}{section}{References}
\begingroup\normalsize
\begin{thebibliography}{42}
\setlength{\itemsep}{0pt}\setlength{\parsep}{0pt}
\bibitem{r1} H. Bahouri, J.-Y. Chemin, and R. Danchin. \emph{Fourier Analysis and Nonlinear Partial Differential Equations}, volume 343 of Grundlehren der Mathematischen Wissenschaften. Springer, Heidelberg, 2011.
\bibitem{r2} R. Chen, Z. Qiao, and S. Zhou. Persistence properties and wave-breaking criteria for the Geng--Xue system. \emph{Math. Methods Appl. Sci.}, 42(18):6999--7010, 2019.
\bibitem{r3} A. Constantin and D. Lannes. The hydrodynamical relevance of the Camassa--Holm and Degasperis--Procesi equations. \emph{Arch. Ration. Mech. Anal.}, 192(1):165--186, 2009.
\bibitem{r4} A. Degasperis, D. D. Holm, and A. N. W. Hone. A new integrable equation with peakon solutions. \emph{Teoret. Mat. Fiz.}, 133(2):170--183, 2002.
\bibitem{r5} A. Degasperis and M. Procesi. Asymptotic integrability. In \emph{Symmetry and Perturbation Theory}. World Sci. Publ., 1999.
\bibitem{r6} J. Dieudonn\'e. \emph{Foundations of Modern Analysis}. Read Books Ltd, 2011.
\bibitem{r7} X. Geng and B. Xue. An extension of integrable peakon equations with cubic nonlinearity. \emph{Nonlinearity}, 22(8):1847--1856, 2009.
\bibitem{r8} A. A. Himonas and C. Kenig. Non-uniform dependence on initial data for the CH equation on the line. \emph{Differential Integral Equations}, 22(3--4):201--224, 2009.
\bibitem{r9} A. A. Himonas, C. Kenig, and G. Misio\l ek. Non-uniform dependence for the periodic CH equation. \emph{Comm. Partial Differential Equations}, 35(6):1145--1162, 2010.
\bibitem{r10} A. A. Himonas and D. Mantzavinos. The initial value problem for a Novikov system. \emph{J. Math. Phys.}, 57(7):071503, 21 pp., 2016.
\bibitem{r11} T. Kato and G. Ponce. Commutator estimates and the Euler and Navier--Stokes equations. \emph{Comm. Pure Appl. Math.}, 41(7):891--907, 1988.
\bibitem{r12} H. Li. Degenerate form of a three-component Degasperis--Procesi equation. \emph{J. Math. Anal. Appl.}, 464(2):1082--1088, 2018.
\bibitem{r13} N. Li. A new 3-component Degasperis--Procesi hierarchy. \emph{Phys. D}, 449:Paper No. 133763, 2023.
\bibitem{r14} Z. Li and Y. Hu. Local well-posedness and blow-up criteria for a three-component Camassa--Holm type equation. \emph{J. Math. Phys.}, 61(2):021501, 2020.
\bibitem{r15} S. Liu and Z. Yin. Blow-up phemomenon for the Geng--Xue system and related models, 2026. arXiv:2601.21295.
\bibitem{r16} H. Lundmark and J. Szmigielski. An inverse spectral problem related to the Geng--Xue two-component peakon equation. \emph{Mem. Amer. Math. Soc.}, 244(1155):viii+87 pp., 2016.
\bibitem{r17} Z. Meng, Y. Nie, W. Ye, and Z. Yin. Global well-posedness, blow-up phenomenon and ill-posedness for the hyperbolic Keller--Segel equations. \emph{J. Differential Equations}, 413:828--850, 2024.
\bibitem{r18} V. Novikov. Generalizations of the Camassa--Holm equation. \emph{J. Phys. A}, 42(34):342002, 14 pp., 2009.
\bibitem{r19} M. Taylor. Commutator estimates. \emph{Proc. Amer. Math. Soc.}, 131(5):1501--1507, 2003.
\bibitem{r20} H. Inci, T. Kappeler, and P. Topalov. On the regularity of the composition of diffeomorphisms. arXiv:1107.0488v2, 2012.
\bibitem{r21} Z. Guo, X. Liu, L. Molinet, and Z. Yin. Ill-posedness of the Camassa--Holm and related equations in the critical space. \emph{J. Differential Equations}, 266:1698--1707, 2019. DOI: 10.1016/j.jde.2018.08.013.
\bibitem{r22} Z. Li, Y. Hu, and X. Wu. Local well-posedness and global existence for a multi-component Novikov equation. \emph{Applicable Analysis}, 100(16):3509--3529, 2021. DOI: 10.1080/00036811.2020.1721477. Preprint: arXiv:1911.11510 (2019).
\bibitem{r24} D.-A. Geba. Ill-posedness in the critical Sobolev space for the Fokas--Olver--Rosenau--Qiao equation. arXiv:2607.13259v1, 2026.
\bibitem{r27} R. Camassa and D. D. Holm. An integrable shallow water equation with peaked solitons. \emph{Phys. Rev. Lett.}, 71(11):1661--1664, 1993.
\bibitem{r28} R. Danchin. A few remarks on the Camassa--Holm equation. \emph{Differ. Integral Equ.}, 14(8):953--988, 2001.
\bibitem{r29} A. Bressan and A. Constantin. Global conservative solutions of the Camassa--Holm equation. \emph{Arch. Ration. Mech. Anal.}, 183(2):215--239, 2007.
\bibitem{r31} R. Camassa, D. D. Holm, and J. M. Hyman. A new integrable shallow water equation. \emph{Adv. Appl. Mech.}, 31:1--33, 1994.
\bibitem{r32} G. M. Coclite and K. H. Karlsen. On the well-posedness of the Degasperis--Procesi equation. \emph{J. Funct. Anal.}, 233(1):60--91, 2006.
\bibitem{r33} A. Constantin, R. I. Ivanov, and J. Lenells. Inverse scattering transform for the Degasperis--Procesi equation. \emph{Nonlinearity}, 23(10):2559--2575, 2010.
\bibitem{r34} J. Escher and Z. Yin. Initial boundary value problems for nonlinear dispersive wave equations. \emph{J. Funct. Anal.}, 256(2):479--508, 2009.
\bibitem{r35} A. A. Himonas and C. Holliman. The Cauchy problem for the Novikov equation. \emph{Nonlinearity}, 25(2):449--479, 2012.
\bibitem{r36} A. A. Himonas and J. Holmes. H\"older continuity of the solution map for the Novikov equation. \emph{J. Math. Phys.}, 54(6):061501, 11 pp., 2013.
\bibitem{r37} A. N. W. Hone and J. P. Wang. Integrable peakon equations with cubic nonlinearity. \emph{J. Phys. A}, 41(37):372002, 10 pp., 2008.
\bibitem{r38} Y. Hou, P. Zhao, E. Fan, and Z. Qiao. Algebro-geometric solutions for the Degasperis--Procesi hierarchy. \emph{SIAM J. Math. Anal.}, 45(3):1216--1266, 2013.
\bibitem{r39} J. Li, Y. Yu, and W. Zhu. Ill-posedness for the periodic Camassa--Holm type equations in critical Besov spaces. \emph{Ann. Mat. Pura Appl.} (4), 205(3):1271--1290, 2026.
\bibitem{r40} Z. Qiao. Integrable hierarchy, $3\times3$ constrained systems, and parametric solutions. \emph{Acta Appl. Math.}, 83(3):199--220, 2004.
\bibitem{r41} Z. Yin. On the Cauchy problem for an integrable equation with peakon solutions. \emph{Illinois J. Math.}, 47(3):649--666, 2003.
\bibitem{r42} Z. Yin. Global weak solutions for a new periodic integrable equation with peakon solutions. \emph{J. Funct. Anal.}, 212(1):182--194, 2004.
\bibitem{r43} {W. Ye, Z. Yin, and Y. Guo. The well-posedness for the Camassa--Holm type equations in critical Besov spaces $B^{1+1/p}_{p,1}$ with $1\le p<+\infty$. \emph{J. Differential Equations}, 367:729--748, 2023. DOI: 10.1016/j.jde.2023.05.032.}
\bibitem{r44} {R. M. Chen, T. Hu, and Y. Liu. The shallow-water models with cubic nonlinearity. \emph{J. Math. Fluid Mech.}, 24(2):Paper No. 49, 31 pp., 2022. DOI: 10.1007/s00021-022-00685-4.}
\bibitem{BerselliBisconti2015}
{L. C. Berselli and L. Bisconti.
On the existence of almost-periodic solutions for the
2D dissipative Euler equations.
\emph{Rev. Mat. Iberoam.}, 31(1):267--290, 2015.
DOI: 10.4171/rmi/833.}
\bibitem{GXPeakonDynamics} {H. Lundmark and J. Szmigielski. Dynamics of interlacing peakons (and shockpeakons) in the Geng--Xue equation. \emph{J. Integrable Syst.}, 2(1):xyw014, 2017. DOI: 10.1093/integr/xyw014.}
\end{thebibliography}
\endgroup
\end{document}